\documentclass[11pt]{article}
 
\usepackage[margin=1in]{geometry} 
\usepackage{amsmath,amsthm,amssymb, graphicx, setspace, mathrsfs, xcolor}

\usepackage{caption}
\usepackage{subcaption}

\usepackage[utf8]{inputenc}
\usepackage{graphicx}

\newcommand{\ec}[1]{\textcolor{purple}{\textbf{[EC: } #1\textbf{]}}}

\let\originalleft\left
\let\originalright\right
\renewcommand{\left}{\mathopen{}\mathclose\bgroup\originalleft}
\renewcommand{\right}{\aftergroup\egroup\originalright}

\newtheorem{thm}{Theorem}[section]

\newtheorem{lem}[thm]{Lemma}
\newtheorem{claim}[thm]{Claim}
\newtheorem{prop}[thm]{Proposition}

\newtheorem{prob}[thm]{Problem}

\newcommand{\ord}{{\operatorname{ord}}}
\newcommand{\ind}{{\operatorname{ind}}}
\newcommand{\aut}{{\operatorname{Aut}}}

\theoremstyle{definition}
\newtheorem{cons}[thm]{Construction}

\newcommand{\eps}{\varepsilon}

\newcommand{\paren}[1]{\left( #1 \right)}

\newcommand{\brac}[1]{\left\{ #1 \right\}}

\newcommand{\Z}{\mathbb{Z}}
\newcommand{\R}{\mathbb{R}}

\newcommand{\C}{\mathscr{C}}

\begin{document}
 
\title{Ordered and colored subgraph density problems}
\author{Emily Cairncross\thanks{Department of Mathematics, Statistics and Computer Science, University of Illinois, Chicago, IL 60607. Email: emilyc10@uic.edu. Research partially supported by NSF Award
DMS-1952767.} \and Dhruv Mubayi\thanks{Department of Mathematics, Statistics and Computer Science, University of Illinois, Chicago, IL 60607. Email: mubayi@uic.edu. Research partially supported by NSF Awards 
DMS-1952767 and DMS-2153576, and by a Simons Fellowship.} }

\maketitle

\begin{abstract}
    We consider three extremal problems about the number of copies of a fixed graph in another larger graph. First, we correct an error in a result of  Reiher and Wagner \cite{RW} and prove that the number of $k$-edge stars  in a graph with  density $x \in [0, 1]$ is asymptotically maximized by a clique and isolated vertices or its complement. Next,  among   ordered $n$-vertex graphs with $m$ edges, we  determine the maximum and minimum number of copies of a $k$-edge star whose nonleaf vertex is minimum among all vertices of the star. Finally, for $s \ge 2$, we define a particular  3-edge-colored complete graph $F$ on $2s$ vertices with colors blue, green and red, and  determine, for each $(x_b, x_g)$ with $x_b+x_g\le 1$ and $x_b, x_g \ge 0$,   the maximum density of $F$ in a large graph whose blue, green and red  edge sets have densities $x_b, x_g$ and $1-x_b-x_g$, respectively. These are the first nontrivial examples of colored graphs for which such complete results are proved.
\end{abstract}

\section{Introduction}

The {\em density} of a graph $G$  with $n$ vertices and $m$ edges is
$\varrho(G):=m / \binom{n}{2}$. For a graph $F$ with $k \leq n$ vertices, let $N(F, G)$ be the number of subgraphs of $G$ that are isomorphic to $F$. We are interested in the minimum and maximum values of $N(F, G)$ over graphs $G$ with a given value of $\varrho(G)$ as $n$ grows. We note that $N(F, G) \le (n)_k/|\aut(F)|$ where $\aut(F)$ is the automorphism group of $G$. 
Define the {\em density of $F$ in $G$} to be \[ \varrho (F, G) := \frac{N(F, G) \cdot |\aut (F)| }{ (n)_k}  \in [0,1]. \]

The classical Kruskal-Katona theorem \cite{13, 14} implies that the maximum density of $K_s$ in a graph of  density $x$ is achieved by a clique and isolated vertices. The minimum density of $K_s$  in a graph with density $x$ was determined by Razborov \cite{Raz} for $s=3$, by  Nikiforov \cite{N} for $s=4$, and  by Reiher \cite{R} for all $s \ge 4$; this is achieved by complete multipartite graphs.  Let $S_k$ denote the $k$-edge star. Ahlswede and Katona \cite{AK} determined the maximum number of $S_2$'s in a graph with density $x$. Reiher and Wagner \cite{RW} claim to prove that the asymptotic maximum value of $N(S_k, G)$ when $G$ has density $x$, is achieved by a clique and isolated vertices or its complement. We correct an error in their proof.  We also consider this problem in the setting of ordered graphs and prove the first nontrivial results on $S_k$ whose vertices have a  particular order.

Finally, we consider similar questions in the induced setting. For $G$ with $n$ vertices and $F$ with $k$ vertices, let $N_\ind (F, G)$ be the number of induced subgraphs of $G$ that are isomorphic to $F$. In other words, $N_\ind (F, G)$ is the number of $S\subset V(G)$ such that $G[S]\cong F$. Let $ \varrho_\ind (F,G) = N_\ind (F, G) / \binom{n}{k} $ be the {\em induced density of $F$ in $G$}. 
In \cite{LMR}, the authors consider the region of possible asymptotic values of $\varrho_\ind (F,G)$ for graphs $G$ of fixed density.  By coloring edges in $G$ and $F$ red, and coloring edges in their complements  blue,  it is apparent that counting induced subgraphs is the same as counting copies of a  two-edge-colored clique in a (larger) two-edge-colored clique. This leads us to consider the maximum asymptotic density of $q$-edge-colored cliques in (larger) $q$-edge-colored cliques with given color densities. For $q = 3$, we prove the first nontrivial result in this setting.

In Section \ref{sec:main}, we  state our results. They are proved in Sections \ref{sec:stars}, \ref{sec:ord}, and \ref{sec:color}. 

\section{Statements of results}\label{sec:main}

\subsection{Stars}
Clearly $N(S_k, G) = \sum_{v \in V(G)} \binom{d(v)}{k}$ and $\varrho (S_k, G) = N (S_k, G) \cdot k! / (n)_{k+1}.$
For $x \in [0,1]$, let $I(S_k,x)$ be the supremum of $\lim_{n \to \infty} \varrho (S_k, G_n)$ over all sequences of graphs $(G_n)_{n=1}^{\infty}$ with $|V(G_n)|\rightarrow \infty$, $\varrho(G_n)\rightarrow x$ and for which $\lim_{n \to \infty} \varrho (S_k, G_n)$ exists.  
For each $\gamma \in [0,1]$, let $\eta=1 - \sqrt{1 - \gamma}$. Then $\gamma^{(k + 1) / 2}$ and $\eta + (1 - \eta) \eta^k$ are  the asymptotic $S_k$-densities in a clique with isolated vertices and the complement of a clique with isolated vertices, both with density $\gamma$. Consequently, 
\[ I(S_k, \gamma) \ge \max \{ \gamma^{(k + 1) / 2},~ \eta + (1 - \eta) \eta^k \}. \]
Reiher and Wagner \cite{RW} proved matching upper bounds on $I(S_k, \gamma)$. Their results are stated (and proved) using the language of graphons, which are limit object of graphs. Formally, a {\em graphon} $W$ is a symmetric, measurable function $W: [0,1]^2 \to [0,1]$   (see Lov\'{a}sz \cite{L} for  background on graphons). For a graphon $W$, let $d_W (x) = \int_0^1 W(x,y) ~dy$ be the {\em degree} of $x$ in $W$, let $t (|, W) = \int_{[0,1]^2} W(x, y) ~dx dy$ be the {\em density} of $W$, and let \[ t (S_k, W) = \int_0^1 d_W^k (x) ~dx \] be the {\em homomorphism density of $S_k$ in $W$}.
\begin{thm} [Reiher-Wagner~\cite{RW}] \label{thm:RW}
    Let $W$ be a graphon and let $k$ be a positive integer. Set $\gamma = t(|, W)$ and $\eta = 1 - \sqrt{1 - \gamma}$. Then \[ t (S_k, W) \le \max \{\gamma^{(k + 1) / 2},~ \eta + (1 - \eta) \eta^k \}. \]
\end{thm}
By the general theory of graphons, Theorem~\ref{thm:RW} gives the same upper bound for $I(S_k, \gamma)$.
There appears to be an error in the proof of \cite[Proposition 3.7]{RW}, which is necessary for the proof of Theorem \ref{thm:RW}. We correct this error and prove Theorem \ref{thm:RW} in Section \ref{sec:stars}.

\subsection{Ordered stars}

 An {\em ordered graph} $G = (V, E)$ is a graph with a total order on $V$. We usually let $V= [n]:=\{1, \ldots, n\}$ with the natural ordering.  When $i<j$ and $\{i,j\} \in E$, we often write $ij \in E$. Let $F$ be an ordered graph on $[s]$. Let $N_{\ord} (F, G)$ be the number of $\{ v_1, \ldots, v_s \} \subseteq [n]$ with $v_1 < v_2 < \cdots < v_s$ such that $v_i v_j \in E (G [v_1, \ldots, v_s])$ whenever $ij \in E(F)$. We consider $N_{\ord} (F, G)$ in the case that $F$ is an appropriate ordered $S_k$. The following constructions provide lower bounds.
 
\begin{cons}\label{cons:quasi-star}
    For positive integers $n$ and $m \le \binom{n}{2}$, let $a$ be the largest integer such that $f(n,a):=\binom{a}{2} + a (n-a) \leq m$. As $f(n,n)={n \choose 2}$, we have $0\le a \le n$.  Set  $b = m - f(n,a)$. Since $f(n, a+1)-f(n,a)=n-a-1$, we conclude that $0\le b <n-a-1$. Let $S_L (n,m)$ be the ordered graph with vertex set $[n]$ and edge set 
    \[\{ vw : v \in [a], w\in [n] \} \cup \{\{a+1, j\}: a+2\le j\le a+b+1\}.\]
    In words,  $S_L (n,m)$ comprises a complete graph on $[a]$, and in addition has all edges between $[a]$ and $[n]\setminus [a]$ and $b$ edges between $a+1$ and the $b$ smallest vertices in $[n]\setminus[a+1]$ (see Figure \ref{fig:quasi-star}). Let $S_R (n,m)$ be defined as $S_L (n,m)$, but where the total order on the vertices is reversed.
\end{cons}

\begin{center}
    \begin{figure}[h]
     \centering
         \includegraphics[scale = 1]{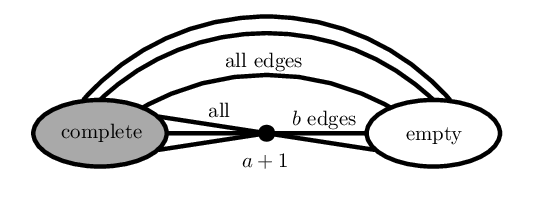}
         \caption{$S_L (n,m)$. }
        \label{fig:quasi-star}
    \end{figure}
\end{center}

Let $S_L (k) := S_L (k+1, k)$ be the ordered {\em left star} and $S_R (k) := S_R (k+1, k)$ the ordered {\em right star}. Note that $S_L(k)$ has $a=1$ and $b=0$ (see Figure \ref{fig:star}). 

\begin{center}
    \begin{figure}[h]
        \centering
        \begin{subfigure}[h]{0.4\textwidth}
            \centering
            \includegraphics[scale = 1]{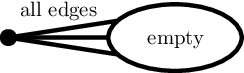}
           
            \label{fig:SL}
        \end{subfigure}
        \begin{subfigure}[h]{0.4\textwidth}
            \centering
            \includegraphics[scale = 1]{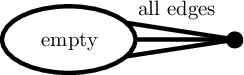}
            
            \label{fig:SR}
        \end{subfigure}
        \caption{$S_L (k)$ and $S_R (k)$.}
        \label{fig:star}
    \end{figure}
\end{center}
Our main result for ordered graphs is the following theorem.

\begin{thm}\label{thm:ord}
    Let $G$ be an ordered graph with vertex set $[n]$ and $m>0$ edges. Then \[ N_\ord (S_L (k), S_R (n,m)) \leq N_\ord (S_L (k), G) \leq N_\ord (S_L (k), S_L (n,m)) .\]
\end{thm}

Theorem~\ref{thm:ord} implies similar results for $S_R(k)$. There are, up to obvious symmetries, $\lceil(k+1)/2\rceil$ different ordered stars with $k$ edges and, apart from $S_L(k)$,  it remains open to prove analogous results to Theorem~\ref{thm:ord} for them.  Indeed, obtaining sharp bounds for these stars seems nontrivial. We address the first open case when $k=2$.

Let $M$ be  the ordered graph with vertex set $[3]$ and edge set $\{12, 23\}$. It seems very difficult to obtain exact results for $N_\ord(M, G)$ so we consider asymptotic growth rates.
Let \[\varrho_\ord (F, G) := \frac{N_{\ord} (F, G) }{ \binom{n}{s}}\]
be the {\em density of $F$ in $G$}. Since the vertices are ordered, each $s$-tuple of vertices can contribute at most one copy of $F$ so $\varrho_\ord (F, G)  \in [0,1]$. Let $(G_n)_{n=1}^\infty$ be a sequence of ordered graphs with $\lim_{n \to \infty} |V(G_n)| = \infty$. The sequence $(G_n)_{n=1}^\infty$ is {\em $F$-good} if both  $x = \lim_{n \to \infty} \varrho (G_n)$  and $y = \lim_{n\to \infty} \varrho_\ord (F, G_n)$ exist. In this case, $(G_n)_{n=1}^\infty$ {\em realizes} $(x,y)$. Define \begin{align*}
    I_{\ord} (F, x) &:= \sup \{ y : (x,y) \in [0,1] \text{ is realized by some $F$-good } (G_n)_{n=1}^\infty \}, \\
    i_{\ord} (F, x) &:= \inf \{ y : (x,y) \in [0,1] \text{ is realized by some $F$-good } (G_n)_{n=1}^\infty \}.
\end{align*}

For any $x\in [0,1]$, Lov\'{a}sz and Simonovits \cite{LS}  constructed a sequence of $K_3$-good complete multipartite graphs $(G_n)_{n=1}^\infty$ such that $\lim_{n \to \infty} \varrho (G_n) = x$. Set 
\[g_3(x):= \lim_{n \to \infty}\varrho (K_3, G_n).\]
Let $V_1, \ldots, V_k$ be the parts of $G_n$ and let $G_n'$ be an ordered graph obtained from $G_n$ with any vertex ordering for which $u<v$ whenever $u \in V_i$, $v \in V_j$ and $i<j$. Then
$N_\ord (M, G_n') = N_\ord (K_3, G_n)$ and this implies that
$i_\ord(M, x) \le g_3(x)$.
Razborov~\cite{Raz} proved that $i_\ord (K_3, x) \ge g_3 (x)$. Note that $M$ is a subgraph of the ordered $K_3$, so $i_\ord(K_3, x) \le i_\ord(M, x)$.  Therefore,
\[g_3(x) \le i_\ord(K_3, x) \le i_\ord(M, x)\] 
and we conclude that $i_\ord (M, x) = i_\ord (K_3, x) = g_3 (x)$. Determining $I_{\ord} (M, x)$ appears to be more difficult.

\begin{cons}\label{cons:spider}
    We construct a sequence of graphs $(P(n,x))_{n = 1}^\infty$ for any $x \in [0,1]$. For each $n$, define $P (n, x)$ to be the ordered graph on $[n] = A \sqcup B \sqcup C$ where $|B| = \lfloor n (1 - \sqrt{1 - x})  \rfloor$, $||A|-|C||\le 1$, and $a < b < c$ for all $a \in A$, $b\in B$, $c\in C$ with edge set \[ \{ ab : a \in A, b \in B \} \cup \{ b_1 b_2 : b_1, b_2 \in B \} \cup \{ bc : b \in B, c \in C \} .\] A short calculation shows that $\lim_{n \to \infty} \varrho ( P(n, x)) = x$ and 
    \[\lim_{n \to \infty} \varrho_\ord (M, P(n, x)) = - \frac{1}{2} \eta (\eta^2 - 3).\]
\end{cons}

\begin{cons}\label{cons:length}
    We construct a sequence of graphs $(Q(n,x))_{n = 1}^\infty$ for any $x \in [0,1]$. For each $n$, define $Q (n, x)$ to be the ordered graph on $[n]$ with edge set \[ \{ ij : j - i \leq \lfloor (1 - \sqrt{1 - x}) n \rfloor \} .\] A short calculation shows that $\lim_{n \to \infty} \varrho ( Q(n, x)) = x$ and  that for $\eta = 1 - \sqrt{1-x}$,  \[\lim_{n \to \infty} \varrho_\ord (M, Q(n, x)) = \begin{cases}
        6 \eta^3 + 6(1 - 2\eta) \eta^2 & \text{ if } \eta \le 1/2, \\
        2\eta^3 - 6 \eta^2 + 6\eta -1 & \text{ if } \eta \ge 1/2.
    \end{cases} \]
\end{cons}

Constructions~\ref{cons:spider} and~\ref{cons:length} show that for  $\eta = 1 - \sqrt{1-x}$,
\[I_{\ord} (M, x) \ge \max \brac{\lim_{n \to \infty} \varrho_\ord (M, P(n, x)),~ \lim_{n \to \infty} \varrho_\ord (M, Q(n, x))}.\]
It is an interesting open problem to determine if the inequality is sharp.

\begin{prob}
    Is
    \[I_{\ord} (M, x) = \max \brac{\lim_{n \to \infty} \varrho_\ord (M, P(n, x)),~ \lim_{n \to \infty} \varrho_\ord (M, Q(n, x))}\]
    for any (possibly all) $x \in [0,1]$?
\end{prob}

A short calculation shows that there exists $x_0 \in [0,1]$ such that $\lim_{n \to \infty} \varrho_\ord (M, Q(n, x)) \le \lim_{n \to \infty} \varrho_\ord (M, Q(n, x))$ iff $x<x_0$. 

\subsection{Colored graphs}\label{sec:intro_color}

A {\em $q$-colored graph} is a graph $G =(V, E)$ together with a coloring function $f: E\to C$ where $|C|=q$. Fix  $q \in \Z^+$ and let $G = (V, E)$ be a $q$-colored complete graph with coloring function $f$. For $1\le i \le q$, let $e_i (G) = | \{ e\in E : f(e) = i \} |$  and let $\varrho_i (G) := e_i (G) / \binom{|V|}{2}$ be the {\em density} of color $i$. Given a $q$-colored complete graph $F$ with $|V(F)|=s$ and coloring function $g$, a subset $X \subset V$ with $|X| = s$ is a {\em copy of $F$ in $G$} if there is a bijection $\sigma : V(F) \to X$ such that $g (uv) = g( \sigma(u) \sigma(v))$ for all distinct $u,v \in V(F)$.  Let $ N_q (F, G )$ be the number of copies of $F$ in $G$ and let $\varrho_q (F, G) := N_q (F, G ) / \binom{|V|}{s}$ be the {\em density of $F$ in $G$}.

Let $(G_n)_{n = 1}^\infty$ be a sequence of $q$-colored complete graphs with $|V(G_n)|\to \infty$. The sequence $(G_n)_{n = 1}^\infty$ is {\em $F$-good} if $x_i = \lim_{n \to \infty} \varrho_i (G_n)$ exists for all $i \in [q]$ and $y = \lim_{n\to \infty} \varrho_q (F, G_n)$ exists. In this case, $(G_n)_{n = 1}^\infty$ {\em realizes} $(x_1, \ldots, x_{q-1}, y)$. Note that we only list $x_1, \ldots, x_{q-1}$ since $x_q = 1 - (x_1+ \cdots + x_{q-1})$. Define \[ I_q (F, (x_1, \ldots, x_{q-1})) := \sup \{ y : (x_1, \ldots, x_{q-1}, y) \in [0,1]^q \text{ is realized by some $F$-good } (G_n)_{n = 1}^\infty \}. \]

 For $2 \le s \le t$, let $K_{s,t}'$ be the $3$-colored clique on vertex set $V = V_1 \sqcup V_2$ with $|V_1| = s$ and $|V_2| = t$ with coloring function $f$ defined by \[f(ij) := \begin{cases}
    \text{blue} & \text{if } i,j \in V_1, \\
    \text{green} & \text{if } i,j \in V_2, \\
    \text{red} & \text{otherwise},
\end{cases}\] for all distinct $i,j \in [s+t]$ (see Figure \ref{fig:K34}).

\begin{center}
    \begin{figure}[h]
        \centering
        \includegraphics[scale = .75]{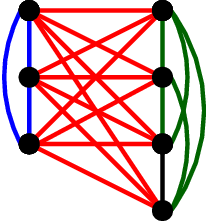}
        \caption{$K_{3,4}'$. }
        \label{fig:K34}
    \end{figure}
\end{center}

\begin{thm}\label{thm:color}
    Let $2 \le s \le t$ and $x_b, x_g, x_r \in [0,1]$ such that $x_b + x_g +x_r = 1$. When $\sqrt{x_b} + \sqrt{x_g} \le 1$, 
    \[I_3 (K_{s,t}', (x_b, x_g)) = x_b^{s/2} x_g^{t/2} \binom{s+t}{s}.
    \] When  $\sqrt{x_b} + \sqrt{x_g} > 1$, \[I_3 (K_{s,s}', (x_b, x_g)) = \paren{\frac{x_r}{2}}^s \binom{2s}{s}.\]
\end{thm}
Theorem~\ref{thm:color} determines $I_3 (K_{s,s}', (x_b, x_g))$ for all vectors $(x_b, x_g)$ in the region $x_b, x_g \ge 0$ and $x_b+x_g \le 1$. Write $\ind (K_{s,s}')$ for the {\em inducibility of $K_{s,s}'$} which is
the maximum value of $I_3 (K_{s,s}', (x_b, x_g))$. An easy optimization shows that
$$\ind (K_{s,s}') = I_3 (K_{s,s}', (1/4, 1/4)) = \paren{\frac{1}{4}}^{s} \binom{2s}{s}.$$ 
Note that $I_3 (K_{s,t}', (x_b, x_g))$ is not known when $\sqrt{x_b} + \sqrt{x_g} > 1$ and $s \neq t$.

\section{Proof of Theorem \ref{thm:RW}}\label{sec:stars}
 Given a function $F: [0,1] \to \mathbb R$, a graphon $W$, and $\gamma \in [0,1]$, let \begin{align*}
    D (F, W) &:= \int_0^1 F (d_W (x)) ~dx, \\
    \operatorname{MAX} (\gamma, F) &:= \max \{ (1 - \sqrt{\gamma}) F (0) + \sqrt{\gamma} F (\sqrt{\gamma}),~ (1 - \eta) F(\eta) + \eta F(1) \}
\end{align*} where $\eta = 1 - \sqrt{1 - \gamma}$.
A measurable function $F: [0,1] \to \R$ is {\em good} for $W$ if \[ 
D(F, W) \le \operatorname{MAX} (\gamma, F)\] where $\gamma = t (|, W)$ and $\eta = 1 - \sqrt{1 - \gamma}$; $F$ is {\em bad} for $W$ if it is not good for $W$.
A collection of measurable functions is good (bad) for $W$ if all its members are good (bad) for $W$.
In \cite{RW}, a set $\C$ of twice differentiable convex functions $F: [0,1] \to \R$ satisfying certain conditions is defined. We do not need the details of the conditions here, so we do not state them, but we note that $\C$ contains the function $F(x)=x^k$.

The error in \cite{RW} appears in the proof of the following proposition.

\begin{prop}[corresponds to {\cite[Proposition 3.7]{RW}}]\label{prop:step}
	$\C$ is good for all step graphons.
\end{prop}

To  prove this proposition, we introduce the following more refined notion of ``good.'' For any $\delta >0$, say that $F \in \C$ is {\em $\delta$-good} for a graphon $W$ if \[ D(F, W) < \operatorname{MAX} (\gamma, F)   + \delta \] for $\gamma = t (|, W)$ and $\eta = 1 - \sqrt{1 - \gamma}$. Say that $F$ is {\em $\delta$-bad} for $W$ if it is not $\delta$-good for $W$.

We next show that several lemmas from \cite{RW} still apply when we replace ``good'' with ``$\delta$-good.'' The proofs are almost exact copies of those in \cite{RW}, and are given in the Appendix.

\begin{lem}[corresponds to {\cite[Lemma 3.2]{RW}}]\label{lem:1-W}
    If all functions in $\C$ are $\delta$-good for a graphon $W$, then the same is true for the graphon $1 - W$.
\end{lem}

Given a graphon $W$ and a real number $\lambda \in [0,1]$, let $[\lambda, W]$ be the graphon satisfying \[ [\lambda, W] (x,y) = \begin{cases}
    0 & \text{ if $0 \le x < \lambda$ or $0 \le y < \lambda$,} \\
    W \paren{\frac{x - \lambda}{1 - \lambda}, ~ \frac{y - \lambda}{1 - \lambda}} & \text{otherwise}.
\end{cases} \]

\begin{lem}[corresponds to {\cite[Lemma 3.5]{RW}}]\label{lem:lambdaW}
    If $\lambda \in [0,1]$ and the graphon $W$ has the property that all functions in $\C$ are $\delta$-good for it, then the same applies to $[\lambda, W]$.
\end{lem}

Similarly, let $[W, \lambda]$ be the graphon satisfying \[ [W, \lambda] (x,y) = \begin{cases}
    W \paren{\frac{x}{1 - \lambda}, ~ \frac{y}{1 - \lambda}} & \text{ if $0 \le x \le 1 - \lambda$ and $0 \le y \le 1 - \lambda$,} \\
    1 & \text{otherwise}.
\end{cases} \]

The next lemma follows from the previous two lemmas and the observation that $[W, \lambda]$ is isomorphic to $1-[\lambda, 1-W]$.

\begin{lem}[corresponds to {\cite[Lemma 3.6]{RW}}]\label{lem:Wlambda}
    If all functions in $\C$ are $\delta$-good for the graphon $W$ and $\lambda \in [0, 1]$, then all functions in $\C$ are good for $[W, \lambda]$ as well.
\end{lem}

We now prove Proposition \ref{prop:step}.
Let $\mathcal G$ be the collection of all step graphons and let $\mathcal G_i$ be the collection of all step graphons with $i$ parts.

\begin{proof}[Proof of Proposition \ref{prop:step}]
    Suppose for a  contradiction that there exists  $W' \in \mathcal G$ and  $F' \in \C$ such that $F'$  is bad for $W'$. This means that $D(F', W') >\operatorname{MAX}(\gamma, W')$. Then there exists $\delta > 0$ such that $D(F', W') \ge \operatorname{MAX}(F', W')+\delta$. In other words, $F'$ is $\delta$-bad for $W'$. Let \[ \mathcal{S} = \mathcal{S}(\delta) := \{ (F,W)  \in \C \times \mathcal{G}: \text{$F$ is $\delta$-bad for }  W \}. \] Note that $(F', W') \in \mathcal{S}$, so $\mathcal{S}\ne \emptyset$ . Partition $\mathcal{S}$ into $\bigcup_{i=1}^\infty \mathcal{S}_i$ where \[ \mathcal{S}_i := \{(F,W)\in \C \times \mathcal {G}_i :\text{$F$ is $\delta$-bad for $W$}\}. \] Let $k$ be the smallest integer such that $\mathcal{S}_k \ne \emptyset$. From \cite[Observation 2.1]{RW} we have $k \ge 2$. Pick $F \in \C$ such that $(F,W) \in \mathcal{S}_k$ for some $W$. Let \[ \mathcal{W} := \{ W : (F, W) \in \mathcal{S}_k \}. \] Let $\mu$ be the Lebesgue measure on $\R$. Each $W \in \mathcal{W}$ is a step function with respect to a partition $\mathcal{P} = \{ P_1, \ldots, P_k \}$ of the unit interval with $\alpha_i = \mu (P_i)$ for all $i \in [k]$ and $\beta_{ij}$ the value attained by $W$ on $P_i \times P_j$ for $i,j \in [k]$. By the choice of $k$, we deduce $\alpha_1, \ldots, \alpha_k >0$ for all $W \in \mathcal{W}$. Define \[ T(W) := \sum_{i = 1}^k \alpha_i (d_W (i))^2 \] where \[ d_W (i) = \int_0^1 W(x,y) ~dy = \sum_{j = 1}^k \alpha_j \beta_{ij} \] for all $i \in [k]$. 
  
    \begin{claim}\label{claim:T}
        $T := \sup_{W \in \mathcal{W}} T(W) = \max_{W \in \mathcal{W}} T(W)$.
    \end{claim}
    
    \begin{proof}
        By the choice of $k$ and $F$, there is a sequence of  step graphons $(W_n)_{n = 1}^\infty$, each with $k$ parts, such that $\lim_{n \to \infty} T (W_n) = T$ and $F$ is $\delta$-bad for $W_n$ for all $n$. Notice that each $W_n$ is defined by the measure of its parts $\alpha_{n,i} \in (0,1)$ for all $i \in [k]$ and its entries $\beta_{n,ij} \in [0,1]$ for all $i,j \in [k]$ (where $\sum \alpha_{n,i} = 1$ and $\beta_{n,ij} = \beta_{n,ji}$). Thus, each $k$-step graphon is defined by $k + \binom{k+1}{2}$ parameters, which we can view as variables. For each of these variables, we use the Bolzano-Weierstrass Theorem to recursively take subsequences of $(W_n)_{n = 1}^\infty$ such that the given variable converges. Call the resulting sequence $(W_n')_{n = 1}^\infty$.
        
        For each $i,j$, let $\alpha_i':=\lim_{n \to \infty} \alpha_{n,i}$ and $\beta_{ij}':=\lim_{n \to \infty}\beta_{n, ij}$. Note that we still have $\sum \alpha_i' = 1$ and $\beta_{ij}' = \beta_{ji}'$. Let $W' \in \mathcal G_k$ have partition $\mathcal{P}' = \{ P_1', \ldots, P_k'\}$ such that $\mu (P_i') = \alpha_i'$ for all $i \in [k]$ and $W'(x,y) = \beta_{ij}'$ whenever $x \in P_i'$ and $y \in P_j'$ for all $i,j \in [k]$.      
        Let $\gamma_n := t (|, W_n')$ and $\gamma := t (|, W')$. We will now prove that $T(W')=T$ and $W' \in \mathcal W$.

        First, observe that
        \[\lim_{n \to \infty}  \sum_{j = 1}^k \alpha_{n,j} \beta_{n,ij} =  \sum_{j = 1}^k  \lim_{n \to \infty} \alpha_{n,j} \beta_{n,ij} = \sum_{j = 1}^k \alpha_j' \beta_{ij}'\] and 
        \[\lim_{n \to \infty} \sum_{i = 1}^k \alpha_{n,i} \paren{\sum_{j = 1}^k a_{n,j} \beta_{n, ij}}^2 = \sum_{i = 1}^k \alpha_{i}' \lim_{n \to \infty} \paren{\sum_{j = 1}^k a_{n,j} \beta_{n, ij}}^2 = \sum_{i = 1}^k \alpha_i' \paren{\sum_{j = 1}^k a_j' \beta_{ij}'}^2\] 
        since all sums are finite and the limit of each summand exists. 
        This gives
        \[ T (W')=\sum_{i = 1}^k \alpha_i' \paren{\sum_{j = 1}^k a_j' \beta_{ij}'}^2 =\lim_{n \to \infty} \sum_{i = 1}^k \alpha_{n,i}  \paren{\sum_{j = 1}^k a_{n,j} \beta_{n, ij}}^2 = \lim_{n \to \infty} T(W_n)=T. \]
        We now show that $W' \in \mathcal W$. As every $F \in \C$ is twice differentiable, it is also continuous and so 
        \[\lim_{n \to \infty} F \paren{\sum_{j = 1}^k \alpha_{n,j} \beta_{n,ij}} =
        F \paren{\lim_{n \to \infty} \sum_{j = 1}^k \alpha_{n,j} \beta_{n,ij}} =
        F \paren{\sum_{j = 1}^k \alpha_j' \beta_{ij}'}.\]
        Since $F$ is continuous on the closed interval $[0,1]$, it is bounded, and we further obtain 
        \[\lim_{n \to \infty} \sum_{i = 1}^k  \alpha_{n,i} F \paren{\sum_{j = 1}^k \alpha_{n,j} \beta_{n,ij}} = \sum_{i = 1}^k \alpha_i' F \left(\sum_{j = 1}^k \alpha_j' \beta_{ij}'\right)=D(F, W').\]
        Consequently, 
        \[\lim_{n \to \infty} D (F, W_n') = \lim_{n \to \infty} \int_0^1 F (d_{W_n'} (x)) ~dx 
        = \lim_{n \to \infty} \sum_{i = 1}^k \alpha_{n,i} F \paren{\sum_{j = 1}^k \alpha_{n,j} \beta_{n,ij}} = D (F, W').\]
        Next, observe that 
        \[ \lim_{n \to \infty} \int_{[0,1]^2} W_n' (x,y) ~dx dy = \lim_{n \to \infty} \sum_{i = 1}^k \sum_{j = 1}^k \alpha_{n, i} \alpha_{n,j} \beta_{n, ij} = \sum_{i = 1}^k \sum_{j = 1}^k \alpha_i' \alpha_j' \beta_{ij}' = \int_{[0,1]^2} W' (x,y) ~dx dy \]
        as all sums are finite and $\lim_{n \to \infty} \alpha_{n, i} \alpha_{n,j} \beta_{n, ij}$ exists for all $i,j \in [k]$. Thus,
        \[ \lim_{n \to \infty} \gamma_n = \lim_{n \to \infty} t (|, W_n') = \lim_{n \to \infty} \int_{[0,1]^2} W_n' (x,y) ~dx dy = \int_{[0,1]^2} W' (x,y) ~dx dy = t (|, W') = \gamma.\]
        Further, since $f(x) = \sqrt{x}$ is continuous on $[0,1]$, \begin{align*}
            \lim_{n \to \infty} \operatorname{MAX} (\gamma_n, F) &= \lim_{n \to \infty} \max \{ (1 - \sqrt{\gamma_n}) F (0) + \sqrt{\gamma_n} F (\sqrt{\gamma_n}),~ (1 - \eta_n) F (\eta_n) + \eta_n F(1) \} \\
            &= \max \{ (1 - \sqrt{\gamma}) F (0) + \sqrt{\gamma} F (\sqrt{\gamma}),~ (1 - \eta) F (\eta) + \eta F(1) \} \\
            &= \operatorname{MAX} (\gamma, F)
        \end{align*} for $\eta_n = 1 - \sqrt{1 - \gamma_n}$ and $\eta = 1 - \sqrt{1 - \gamma}$. Finally, as $F$ is $\delta$-bad for $W_n'$, 
        \[D (F, W') =\lim_{n \to \infty} D (F, W_n') \ge \lim_{n \to \infty} \operatorname{MAX} (\gamma_n, F) +\delta =\operatorname{MAX} (\gamma, F) + \delta.\] This means that $F$ is $\delta$-bad for $W'$. Therefore, $W' \in \mathcal W$ and the claim is proved. 
    \end{proof}

    Fix $W \in \mathcal{W}$ with partition $\mathcal{P} = \{P_1, \ldots, P_k\}$ and parameters $\alpha_i$  and $\beta_{ij}\in [0,1]$ for all $i,j\in [k]$ such that $T(W) = T$. Recall that by the minimality of $k$, we know that $\alpha_i \in (0,1)$. By definition of $\mathcal W$, $F$ is $\delta$-bad for $W$. Set $d_i := d_W (i)$.
    To obtain the necessary contradiction to complete the proof of the proposition,  we will show that $F$ is $\delta$-good for $W$, i.e. that 
    \[ D(F, W) = \int_0^1 F (d_W (x)) ~dx = \sum_{i = 1}^k \alpha_i F (d_i) < \operatorname{MAX} (\gamma, F) + \delta, \] 
    for $\gamma =  t (|, W)= \sum_{i = 1}^k \alpha_i d_i$. Without loss of generality, we may assume $d_1 \le d_2 \le \cdots \le d_k$.
    
    \begin{claim}[corresponds to {\cite[Claim 3.8]{RW}}] \label{claim:main}
    	If $1 \le r < s \le k$ and $\beta_{ir} > 0$, then $\beta_{is} = 1$.
    \end{claim}

    \begin{proof}[Proof of Claim \ref{claim:main}]
        Suppose, for contradiction, that $\beta_{ir} > 0$ and $\beta_{is} < 1$.
        Define  $Q \in \mathcal G_k$ with the same partition $\mathcal{P}$ as follows: let $\delta_{ij}$ denote the Kronecker delta, and set, for $x \in P_m$ and $y \in P_n,$
        \[ Q(x,y) = \begin{cases}
            -(1 + \delta_{ir}) \alpha_s & \text{ if } \{m, n\} = \{i, r\}, \\
            (1 + \delta_{is}) \alpha_r & \text{ if } \{m, n\} = \{i, s\}, \\
            0 & \text{ otherwise.}
        \end{cases} \]
        Let $\eps \ge 0$ be maximal such that $W_{\varepsilon} = W + \eps Q$ still satisfies $W_{\varepsilon} (x,y) \in [0,1]$ for all $x,y \in [0,1]$. By our assumptions on $\beta_{ir}$ and $\beta_{is},$  we know that $\eps>0$.

        For all $j \in [k]$, let $d_j'$ denote the value attained by $d_{W_{\varepsilon}} (x)$ for all $x \in P_j$.  We have
        \[d_r' - d_r = -(1 + \delta_{ir}) \alpha_i \alpha_s \eps + \delta_{ir} (1 + \delta_{is}) \alpha_i \alpha_s \eps = - \alpha_i \alpha_s \eps, \] and similarly \[ d_s' - d_s = \alpha_i  \alpha_r \eps. \]
        Further, if $i \not\in\{r,s\}$, then
        \[d_i'-d_i=\alpha_r(\beta_{ir}-\alpha_s)+\alpha_s(\beta_{is}+\alpha_r) - (\alpha_r\beta_{ir} +\alpha_s\beta_{is})=0.\]
        Therefore $d_j' = d_j$  for all $j \notin \{r,s\}$. Consequently,
        \begin{align*}
            T(W_{\varepsilon}) - T(W) =   \sum_{j = 1}^k \alpha_j (d_j')^2 - \sum_{j = 1}^k \alpha_j d_j^2 &= \alpha_r (d_r - \alpha_i \alpha_s \eps)^2 - \alpha_r d_r^2 + \alpha_s (d_s + \alpha_i \alpha_r \eps)^2 - \alpha_s d_s^2 \\
            &= 2 \eps\alpha_i \alpha_r \alpha_s (d_s - d_r) + \eps^2 \alpha_i^2\alpha_r \alpha_s(\alpha_r+\alpha_s) \\
            &\ge \eps^2 \alpha_i^2\alpha_r \alpha_s(\alpha_r+\alpha_s) >0
        \end{align*} since $\alpha_i, \alpha_r, \alpha_s, \eps > 0$ and $d_s \ge d_r$ by assumption. Thus $T(W_{\varepsilon}) > T(W)$, and clearly $W_{\eps} \in \mathcal G_k$, so we conclude by the choice of $W$ that $W_{\eps} \not\in \mathcal W$.
        To complete the proof, we will now obtain the contradiction $W_{\eps} \in \mathcal W$ by showing that $F$ is $\delta$-bad for $W_{\eps}$. 
        First, observe that  \[ \int_{[0,1]^2} Q (x,y) ~dx dy = \alpha_i \alpha_r \alpha_s ((1 + \delta_{is}) (2 - \delta_{is}) - (1 + \delta_{ir}) (2 - \delta_{ir})) = 0. \] This implies that $t (|, W_{\varepsilon}) = t (|, W) = \gamma$. Next,
        \begin{align*}
            D(F, W_\varepsilon) - D(F, W) &= \int_0^1 F (d_{W_{\varepsilon}} (x)) ~ dx  - \int_0^1 F (d_{W} (x)) ~ dx \\
            &= \sum_{j = 1}^k \alpha_j (F (d_j') - F (d_j)) \\
            &= \alpha_s (F (d_s + \alpha_i \alpha_r \eps) - F (d_s)) + \alpha_r (F (d_r - \alpha_i \alpha_s \eps) - F (d_r)) \\
            &\ge \alpha_i \alpha_r \alpha_s \eps (F' (d_s) - F' (d_r)) \ge 0
        \end{align*} by convexity of $F$ and  $d_s \ge d_r$.
        Since $F$ is $\delta$-bad for $W$, we obtain 
        \[D(F, W_\varepsilon) \ge D(F, W) \geq \operatorname{MAX} (\gamma, F) + \delta.\]
        This shows that $F$ is $\delta$-bad for $W_\eps$ and completes the proof of the claim.
    \end{proof}
  
    \begin{claim}[corresponds to {\cite[Claim 3.9]{RW}}] \label{claim:beta0}
        $\beta_{1k} >0$.
    \end{claim}

    \begin{proof}
        If this does not hold, then $\beta_{1k} = 0$ and the previous claim implies $\beta_{1i} = 0$ for all $i \in [k-1]$. It follows that there exists $W' \in \mathcal G_{k-1}$ such that $W$ is isomorphic to $[\alpha_1, W']$. Due to the minimality of $k$ all functions in $\C$ are $\delta$-good for $W'$ and by Lemma \ref{lem:lambdaW} the same applies to the graphon $W$, contrary to its choice.
    \end{proof}

    \begin{claim}[corresponds to {\cite[Claim 3.10]{RW}}] \label{claim:beta1}
        $\beta_{1k} <1$.
    \end{claim}

    \begin{proof}
        Suppose for contradiction that $\beta_{1k}=1$. Then $\beta_{k1} = \beta_{1k}>0$, and Claim \ref{claim:main} implies $\beta_{ki} = 1$ for all $1\le i \le k$. So some  $W' \in \mathcal G_{k-1}$ has the property that $[W', \alpha_k]$ is isomorphic to $W$. This implies that all functions in $\C$ are $\delta$-good for $W'$ and then Lemma \ref{lem:Wlambda} implies that all functions in $\C$ are $\delta$-good for $W$, contradiction.
    \end{proof}

    By Claims~\ref{claim:beta0} and \ref{claim:beta1}, we have $0 < \beta_{1k} < 1$. Moreover, by Claim \ref{claim:main},  $\beta_{1i} = 0$ for all $i \in [k-1]$ and $\beta_{jk} = 1$ for all $j$ with $2 \le j \le k$. Divide $P_k$ into two measurable subsets $Q_k$ and $Q_{k + 1}$ satisfying $\mu (Q_k) = (1 - \beta_{1k}) \alpha_k$ and, consequently, $\mu (Q_{k + 1}) = \beta_{1k} \alpha_k$. Set $Q_i = P_i$ for $i \in [k-1]$ and $\mathcal{Q} = \{ Q_1, \ldots, Q_{k + 1} \}$. Let $W'$ be the step graphon with respect to $\mathcal{Q}$ defined as follows: for $x \in Q_i$ and $y \in Q_j$, \[ W' (x,y) = \begin{cases}
        \beta_{ij} & \text{ if $2 \le i \le k$ and $2 \le j \le k$,} \\
        0 & \text{ if $i=1$ and $j \in [k]$ or vice versa,} \\
        1 & \text{ if $i = k+ 1$ or $j = k + 1$.}
    \end{cases} \] By the last two clauses $W'$ is isomorphic to a graphon of the form $[[\alpha_1 / (1 - \beta_{1k} \alpha_k), W''], \beta_{1k} \alpha_k]$ for some graphon $W''$, and by the first clause $W'' \in \mathcal{G}_{k-1}$. So Lemma \ref{lem:lambdaW} and Lemma \ref{lem:Wlambda} show that $F$ is $\delta$-good for $W'$. Since $t (|, W') = t (|, W') = \gamma$, Jensen's inequality implies that
    \[ D (F, W) = \sum_{i = 1}^k \alpha_i F (d_i) \leq \sum_{i = 1}^{k - 1} \alpha_i F (d_i) + \alpha_k ((1 - \beta_{1k}) F (d') + \beta_{1k} F (d'')) = D (F, W') \]
    for some real numbers $d'$ and $d''$ satisfying $(1 - \beta_{1k}) d' + \beta_{1k} d'' = d_k$. Therefore, $D (F, W) \le D (F, W') < \operatorname{MAX} (\gamma, F) + \delta$ so $F$ is $\delta$-good for $W$, a contradiction.
\end{proof}

The proof of Theorem \ref{thm:RW} follows from Proposition \ref{prop:step} exactly as in \cite{RW}.
\medskip

The error in~\cite{RW} is  in the proof of Claim~\ref{claim:main}.
In~\cite{RW}, $T$ is defined as the number of pairs $(i, j) \in [k]^2$ for which $\beta_{ij} \in \{0, 1\}$ and $W\in \mathcal G_k$ is chosen to maximize $T$. Then Claim 3.8 in~\cite{RW} states that $\beta_{ir} > 0$ and $\beta_{is} < 1$ together imply that $W_{\varepsilon}$ is  $0$ or $1$ on at least $T+1$ of the sets $P_i \times P_j$ by construction. However, this is not true if $\beta_{ir} = 1$ and $\beta_{is} = 0$. For example, consider the graphon $W$ defined on the partition $P_1 = [0, 2/5), P_2 = [2/5, 4/5), P_3 = [4/5, 1]$ by \[
W (x,y) = \begin{cases}
    1 & \text{ if } (x,y) \in (P_1 \times P_3) \cup (P_2 \times P_2)\cup (P_2 \times P_3) \\
    0 & \text{ if } (x,y) \in (P_1 \times P_1) \cup (P_1 \times P_2) \cup (P_3 \times P_3)
\end{cases}
\] (see Figure \ref{fig:counterexample}). We see that $d_1 = 1/5$, $d_2 = 3/5$, and $d_3 = 4/5$, so $d_1 \le d_2 \le d_3$ is satisfied. Setting $i=3$, $r=2$, and $s=3$, we have that $\beta_{ir} =1 >0$ and $\beta_{is} =0 < 1$.
However, all entries are already $0$ or $1$ in $W$, so $W_\eps$ cannot equal $0$ or $1$ on more sets than $W$. Note that any step graphon with parts ordered by degree such that $\beta_{ir} =1$ and $\beta_{is} =0$ for some $i$ and $r<s$ is also a counterexample; the other entries need not equal $0$ or $1$ as in this example.

\begin{figure}[h]
    \centering
    \includegraphics[scale=.5]{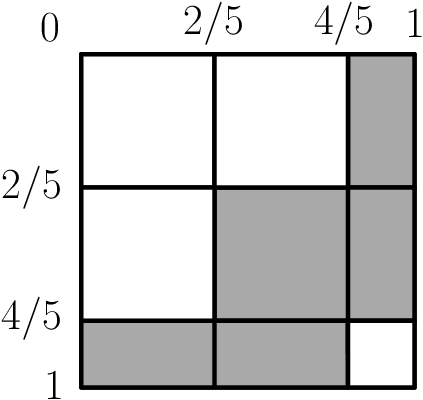}
    \caption{A counterexample to \cite{RW}, Claim 3.8 (shaded squares have $\beta$ value 1).}
    \label{fig:counterexample}
\end{figure}
    
\section{Proof of Theorem \ref{thm:ord}}\label{sec:ord}

Recall that we are assuming $G=([n], E)$.  For a vertex $i$ in  $G$, the {\em right-degree} of $i$ is~$d^+_{G}(i):=|\{j>i: ij \in E\}|$.  Recall that when we write $uv$ for an edge, we implicitly mean $u<v$.
The {\em left vertex} of $vw \in E$ is $v$.  A vertex $v$ has {\em full right-degree} if $E \supset \{vw:w>v\}$. Note that $N_\ord (S_L (k), G) = \sum_{i \in V(G)} \binom{d_G^+ (i)}{k}.$

\begin{proof}[Proof of Theorem \ref{thm:ord}]
    If $m=0$, the theorem is trivial, so we assume that $m >0$.

    Let $G=([n], E)$ with $|E|=m>0$ such that $G$ maximizes the number of copies of $S_L (k)$ among all ordered $n$ vertex $m$ edge graphs. Suppose that there exists $i \in [n]$ such that $d^+_G (i) < d^+_G (i+1)$.
    Let $G_i$ be the ordered graph obtained from $G$ by interchanging the positions of $i$ and $i+1$ in the ordering of $V(G)$. We say that $G_i$ is obtained from $G$ by {\em swapping} $i$ and $i+1$. 
   
    Since $i$ and $i+1$ are consecutive in the ordering of $G$, we have that $d_G^+ (j) = d_{G_i}^+ (j)$ for all $j \in [n] \setminus \{ i, i+1\}$. If there is no edge between $i$ and $i+1$, then $d_G^+ (i) = d_{G_i}^+ (i)$ and $d_G^+ (i+1) = d_{G_i}^+ (i+1)$ as well, so  $N(S_L (k), G_i)=N(S_L(k), G)$. If there is an edge between $i$ and $i+1$, then $d_{G_i}^+(i)=d_G^+(i)-1$ and $d_{G_i}^+(i+1)=d_G^+(i+1)+1$. By convexity of the binomial coefficient and the fact that $d_G^+ (i) < d_G^+ (i+1)$, \begin{align*}
        N_\ord (S_L (k), G_i) ~ -& ~ N_\ord (S_L (k), G) = \sum_{j \in V(G_i)} \binom{d_{G_i}^+ (j)}{k} - \sum_{j \in V(G)} \binom{d_G^+ (j)}{k} \\
        &= \paren{\binom{d_{G_i}^+ (i)}{k} + \binom{d_{G_i}^+ (i + 1)}{k}} - \paren{\binom{d_G^+ (i)}{k} + \binom{d_G^+ (i + 1)}{k}} \\
        &= \paren{\binom{d_G^+ (i) - 1}{k} + \binom{d_G^+ (i + 1)+ 1}{k}} - \paren{\binom{d_G^+ (i)}{k} + \binom{d_G^+ (i + 1)}{k}} > 0.
    \end{align*} This contradicts our choice of $G$, so there is no $i \in [n]$ such that $d_G^+ (i) < d_G^+ (i+1)$ and $\{i,i+1\} \in E$.  We repeatedly swap consecutive pairs of vertices $v< w$ if the right degree of $v$ is less than the right degree of $w$ in the current graph until no such pairs $v,w$ exist. Call the resulting graph $G'$. Then  $d^+_{G'} (v)  \ge d^+_{G'} (w)$ for all $v<w$ and $N_\ord (S_L (k), G') = N_\ord (S_L (k), G)$.
    
    Let $v := \min \{ i \in [n] : d_{G'}^+ (i) < n - i \}$ and choose  $w\in [n]$  such that $vw \not\in E(G')$.
    Let $x := \max \{ i\in [n] : d^+_{G'} (i) > 0 \}$.  We will show that $v=x$ or $x+1$. Note that if $v = 1$, then $x\ge 1 = v$; if $v\ge 2$, then the choice of $v$ implies that $d^+_{G'} (v-1) = n - v + 1 >0$, so $x \ge v-1$. Thus $v \le x+1$ and it remains to show that $v \geq x$. Fix $y\in [n]$ such that $xy \in E (G')$.
    
    Suppose, for contradiction, that $v < x$. Let $H = ([n], E(G') \cup vw \setminus xy)$. Then $d^+_H(v)=d^+_{G'}(v)+1$ and $d^+_H(x)=d^+_{G'}(x)-1$.
    By convexity of the binomial coefficient and the fact that $d_{G'}^+ (v) \ge d_{G'}^+ (x)$, 
    \[N_\ord (S_L (k), H) > N_\ord (S_L (k), G')= N_\ord (S_L (k), G). \] This contradicts our choice of $G$, so it must be the case that $v=x$ or $x+1$.  This means that \[ E(G') = \{ ij : i \in [v-1], j\in [n] \} \cup A \] for some $A \subset \{ vj : j \in \{v+1, v+2, \ldots, n\} \}$ with $|A| = b$ for some $0 \le b < n-v$. Since $d^{+}_{G'}(v)=b$ and $d^+_{G'}(w)=0$ for $w>v$, 
    \[N_\ord (S_L (k), G) = N_\ord (S_L (k), G')
    =N_\ord (S_L (k), S_L (n,m))\]
    as required.

    Let $G=([n], E)$ with $|E|=m$ such that $G$ minimizes the number of copies of $S_L (k)$ among all ordered $n$ vertex $m$ edge graphs. Suppose further that there are some $i,j \in V(G)$ with $i < j$ such that $d^+_G (j) < n-j$ and $d^+_G(i) > d^+_G(j)$. Then choose $v \in [n]$ such that $j < v$ and $jv \notin E$. As $d^+_G(i) > d^+_G (j) \ge 0$, we also choose $w \in [n]$ such that $i < w$ and $iw \in E$. Let $H = ([n], E \cup jv \setminus iw)$. Then $d^+_H(i)=d^+_{G}(i)-1$ and $d^+_H(j)=d^+_{G}(j)+1$. By convexity of the binomial coefficient, 
    \[N_\ord (S_L (k), H)\leq N_\ord (S_L (k), G),\] where equality holds exactly when $d_G^+(i) = d_G^+(j) + 1$. By our assumption on $G$, we must have $N_\ord (S_L (k), H) = N_\ord (S_L (k), G)$. We repeatedly remove an edge $xu$ and add an edge $yz$ for $x < y$ satisfying $d^+ (y) < n-y$ and $d^+(x) > d^+(x)$ until no such pairs $x,y$ exist.
    Call the resulting graph $G'$ and note that $G'$ has $m$ edges. Then $N_\ord (S_L (k), G') = N_\ord (S_L (k), G)$ and \[\text{for all $i < j$ in $G'$, either $d^+_{G'} (j) = n - j$ or $d^+_{G'} (i) \le d^+_{G'} (j)$} \qquad(*).\] 
    We note here that the right degree sequence of $S_R(n,m)$ is given by \begin{equation}\label{eqn:deg}
        0,1,\ldots, a-1, \underbrace{a, \ldots, a}_{n-a-b}, \underbrace{a+1, \ldots, a+1}_{b}.
    \end{equation}
    Our plan is to show that the right degree sequence of $G'$ is the same as (\ref{eqn:deg}). This will allow us to conclude that $N_\ord (S_L (k), G') =N_\ord (S_L (k), S_R (n,m))$.
    
    Set $v := \min \{ j \in [n] : d^+_{G'} (j) = n-j \}$. Then $d^+_{G'} (i) \le d^+_{G'} (j)$ for all $i < j \le v-1$ by $(*)$. Furthermore,  if there exists $i\ge v$ such that $d^+_{G'} (i) < n - i$, then we can find consecutive vertices $x,y$ such that $v \le x<y$, $d^+_{G'}(y)<n-y$
    and $d^+_{G'}(x)=n-x$. But this is impossible as ($*$) implies that  $d^+_{G'}(y)\ge d^+_{G'}(x)=n-x>n-y$.
    Consequently, $d^+_{G'} (i) = n - i$
    for all $i \in \{ v, v+1, \ldots, n \}$. 
    
    If $v=n$, then $d^+_{G'} (n-1) =0$ and $d^+_{G'} (i) \le d^+_{G'} (n-1) = 0$ for all $i \in [n-1]$, so $|E(G')| = 0$. This contradicts $m>0$, so we must have $v < n$. If  $d^+_{G'} (i) \geq n-v +1$ for some $i \in [v-1]$, then since $i \leq v-1$, we have $d^+_{G'} (v-1) \geq d^+_{G'} (i) \ge n-v+1 = n - (v - 1)$. This implies that $d^+_{G'}(v-1)=n-(v-1)$, contradicting the choice of $v$. Therefore
    $v < n$ and  $d^+_{G'} (i)\le  n-v$ for all $i \in [v-1]$. 
    
    Next, suppose that $d^+_{G'} (1) < n-v-1$. Note that $d^+_{G'} (v) = n-v > n-v-1 > d^+_{G'} (1)$ implies that $d^+_{G'} (v) \ge d^+_{G'} (1) + 2$. Choose $x \in [n]\setminus \{ 1 \}$ such that $1x \notin E(G')$ and $y \in \{ v+1, \ldots, n \}$ such that $vy \in E(G')$. Let $H' = ([n], E(G') \cup 1x \setminus vy)$. Then $d^+_{H'} (1)=d^+_{G'} (1) - 1$ and $d^+_{H'} (v)=d^+_{G'} (v) + 1$. By convexity of  binomial coefficients and the fact that $d^+_{G'} (v) \geq d^+_{G'} (1)+2$, 
    \[N_\ord (S_L (k), H') < N_\ord (S_L (k), G') = N_\ord (S_L (k), G).\]
    This contradicts the choice of $G$, so $n-v-1\le d^+_{G'} (1) \le n-v$.
    
    Since the right-degrees of vertices $x$ are nondecreasing for $1\le x \le v$,  there is some $0 < b_0 \le v$ such that $d^{+}_{G'}(i)= n-v - 1$ for all $i \in [v-b_0]$, $d^{+}_{G'}(i)= n-v$ for the remaining $b_0$ vertices $v-b+1 \le i \le v$; we have already observed that $d^{+}_{G'}(i)= n-i$ for all $i > v$. Thus the right degree sequence of $G'$ is given by \begin{equation}\label{eqn:deg2}
        0,1,\ldots, n-v-2, \underbrace{n-v-1, \ldots, n-v-1}_{v+1-b_0}, \underbrace{n-v, \ldots, n-v}_{b_0}.
    \end{equation} If $b_0 = v$, set $b=0$ and $a=n-v$. Otherwise, set $b = b_0$ and $a = n - v- 1$. In either case, the degree sequence of $S_R (n,m)$ in (\ref{eqn:deg}) is the same as that of $G'$ in (\ref{eqn:deg2}) and therefore 
    \[N_\ord (S_L (k), G) 
    = N_\ord (S_L (k), G')
    = N_\ord (S_L (k), S_R (n,m))\] as required.
\end{proof}

\section{Proof of Theorem \ref{thm:color}}\label{sec:color}

Let $(G_n)_{n=1}^\infty$ be a sequence of graphs with $\lim_{n \to \infty} |V(G_n)| = \infty$. The sequence $(G_n)_{n=1}^\infty$ is {\em $F$-good} if both  $x = \lim_{n \to \infty} \varrho (G_n)$  and $y = \lim_{n\to \infty} \varrho_\ind (F, G_n)$ exist. In this case, we say that $(G_n)_{n=1}^\infty$ {\em realizes} $(x,y)$. Define \[ I_{\ind} (F, x) := \sup \{ y : (x,y) \in [0,1] \text{ is realized by some $F$-good } (G_n)_{n=1}^\infty \}. \] 
The following two results will be needed to give upper bounds.

\begin{thm}[Kruskal-Katona \cite{13, 14}]\label{thm:KK}
    Let $r \ge 2$ be an integer. Then for every $x \in [0, 1]$ we have $I_\ind (K_r, x) \le x^{r/2}$. 
\end{thm}

\begin{thm}[Liu, Mubayi, Reiher \cite{LMR}]\label{thm:LMR}
    Let $s \ge 2$ be an integer. Then for every $x \in [0, 1]$ we have $I_\ind (K_{s,s}, x) \le \binom{2s}{s} x^s / 2^s$.
\end{thm}

\begin{proof}[Proof of Theorem~\ref{thm:color}]

    \noindent
    We address the two cases separately. 
    \medskip
        
    \noindent{\bf Case 1:}   $\sqrt{x_b} + \sqrt{x_g} \le 1$.
    
    \noindent We first prove the upper bound.  Suppose that $(G_n)_{n = 1}^\infty$ be a $K_{s,t}'$-good sequence that realizes $( (x_b, x_g), I_3 (K_{s,t}', (x_b, x_g)) )$ and let $f_n$ be the coloring function for $G_n$. Let $K_s' = ([s], \binom{s}{2})$ with coloring function $f_b \equiv \text{blue}$. Similarly, let $K_t' = ([t], \binom{t}{2})$ have coloring function $f_g \equiv \text{green}$. Since $K_{s,t}'$ contains exactly one copy of $K_s'$ and one copy of $K_t'$,
    \begin{equation} \label{eqn:st}
        N_3 (K_{s,t}', G_n) \leq N_3 (K_{s}', G_n) N_3 (K_{t}', G_n).
    \end{equation}

    By deleting edges that are not colored blue and applying Theorem \ref{thm:KK}, we obtain \[ \lim_{n \to \infty} 
    \frac{N_3 (K_{s}', G_n)}{\binom{|V (G_n)|}{s}} \le I_3 (K_s', (x_b, x_g)) \le I_\ind (K_s, x_b) \le x_b^{s/2}. \]  Similarly, by deleting edges not colored green, we obtain \[ \lim_{n \to \infty} \frac{N_3 (K_{t}', G_n)}{\binom{|V (G_n)|}{t}} \le I_3 (K_t', (x_b, x_g)) \le I_\ind (K_t, x_g) \le x_g^{t/2}. \] Together with (\ref{eqn:st}), we get 
    \[I_3 (K_{s,t}', (x_b, x_g)) 
    \le \lim_{n \to \infty} \frac{x_b^{s/2} \binom{|V (G_n)|}{s} \cdot x_g^{t/2} \binom{|V (G_n)|}{t}}{\binom{|V (G_n)|}{s+t}} 
    = x_b^{s/2} x_g^{t/2} \binom{s+t}{s}.\]
    We now prove the lower bound. As $
        \lfloor n \sqrt{x_b} \rfloor + \lfloor n \sqrt{x_g} \rfloor \le n \sqrt{x_b} + n \sqrt{x_g} \le n,
    $ there exist disjoint subsets  $A_n$ and  $B_n$ of $[n]$ with  $|A_n| = \lfloor n \sqrt{x_b} \rfloor$ and $|B_n| = \lfloor n \sqrt{x_g} \rfloor$ for all $n$. For each $n \ge 1$, let  $G_n = ([n], \binom{[n]}{2})$ have coloring function $f_n$ defined by \[ f_n (ij) = \begin{cases}
        \text{blue} & \text{ if } i,j \in A_n, \\
        \text{green} & \text{ if } i,j \in B_n, \\
        \text{red} & \text{ otherwise.} 
    \end{cases} \]
    \begin{figure}[h]
        \centering
        \includegraphics[scale=1]{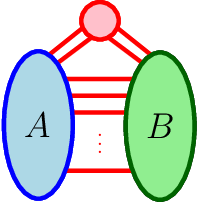}
        \caption{Construction maximizing $K_{s,t}'$-density when $\sqrt{x_b} + \sqrt{x_g} \le 1$.}
        \label{fig:K4construction1}
    \end{figure}
    \noindent 
    Clearly, $
    \lim_{n\to \infty} \varrho_{\text{blue}} (G_n)  = x_b$ and $\lim_{n\to \infty} \varrho_{\text{green}} (G_n) = x_g$. If $A$ is a copy of $K_s'$ in $G_n$ and $B$ is a copy of $K_t'$ in $G_n$, then $A \cup B$ is a copy of $K_{s,t}'$ in $G_n$, as $f_n (ij) = \text{red}$ for all $i \in A_n$, $j \in B_n$. Thus, $(G_n)_{n = 1}^\infty$ realizes $(x_b, x_g, y)$ where \[y = \lim_{n \to \infty} \varrho_3 (K_{s,t}', G_n)= \lim_{n \to \infty} \frac{\binom{\lfloor n \sqrt{x_b} \rfloor}{s} \cdot \binom{\lfloor n \sqrt{x_g} \rfloor}{t}}{\binom{n}{s + t}} 
    = \lim_{n \to \infty} \frac{ (s + t)! n^s x_b^{s/2} n^t x_g^{t/2}}{n^{s + t} s! t!} 
    = x_b^{s/2} x_g^{t/2} \binom{s+t}{s}.\] Consequently, 
    $ I_3 (K_{s,t}', (x_b, x_g)) \ge x_b^{s/2} x_g^{t/2} \binom{s+t}{s}$.

    \medskip
    
    \noindent{\bf Case 2:}    $\sqrt{x_b} + \sqrt{x_g} > 1$.
    
    \noindent 
    We first prove the upper bound. Set $x_r := 1 - x_b - x_g$.  Let $(G_n)_{n = 1}^\infty$ be a $K_{s,s}'$-good sequence that realizes $( (x_b, x_g), I_3 (K_{s,s}', (x_b, x_g)) )$. 
    Every copy of $K_{s,s}'$ contains a red copy of $K_{s,s}$. Consequently, Theorem \ref{thm:LMR} implies that
    \[  I_3 (K_{s,s}', (x_b, x_g)) = \lim_{n \to \infty} \frac{N_3 (K_{s,s}', G_n)}{\binom{|V (G_n)|}{2s}}  \le I_\ind (K_{s,s}, x_r) \le \paren{\frac{ x_r}{2}}^s \binom{2s}{s}. \]
    We now prove the lower bound. Note that
    \[\lfloor n \sqrt{x_b} \rfloor + \left\lfloor n \frac{x_r}{2 \sqrt{x_b}} \right\rfloor \le 
    \frac{2x_b + x_r }{2 \sqrt{x_b}} n
    = \frac{1 +x_b - x_g }{2 \sqrt{x_b}} n 
    < \frac{2 \sqrt{x_b}}{2 \sqrt{x_b}} n = n\] since $\sqrt{x_b} + \sqrt{x_g} > 1$ implies $x_g > 1 - 2 \sqrt{x_b} + x_b$. For each $n$, let $A_n \sqcup B_n \subseteq [n]$ where $|A_n| = \lfloor n \sqrt{x_b} \rfloor$ and $|B_n| = \lfloor n x_r/ (2 \sqrt{x_b}) \rfloor$ and let  $G_n = ([n], \binom{[n]}{2})$ have coloring function $f_n$ defined by 
    \[ f_n (ij) = \begin{cases}
        \text{blue} & \text{ if } i,j \in A_n, \\
        \text{red} & \text{ if } i \in A_n, j \in B_n \text{ or } j \in A_n, i\in B_n, \\
        \text{green} & \text{ otherwise.}
    \end{cases} \]
    \begin{figure}[h]
        \centering
        \includegraphics[scale=1]{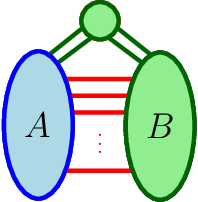}
        \caption{Construction maximizing $K_{s,s}'$-density when $\sqrt{x_b} + \sqrt{x_g} > 1$.}
        \label{fig:K4construction2}
    \end{figure}
    Now $\lim_{n \to \infty} \varrho_{\text{blue}} (G_n) = x_b$, and 
    \[ \lim_{n \to \infty} \varrho_{\text{red}} (G_n) = \lim_{n\to \infty} \frac{\lfloor n \sqrt{x_b} \rfloor \cdot \lfloor n x_r / 2 \sqrt{x_b} \rfloor}{\binom{n}{2}} = \frac{2 x_r \sqrt{x_b}}{2 \sqrt{x_b}} = x_r. \]
    We conclude  that $\lim_{n \to \infty} \varrho_{\text{green}} (G_n) = 1-x_b-x_r=x_g$. Every pair of $s$ vertices in $A_n$ and $s$ vertices in $B_n$ induces a copy of $K_{s,s}'$, so $(G_n)_{n = 1}^\infty$ realizes $( x_b, x_g, y )$ where 
    \[ y = \lim_{n\to \infty} \frac{\binom{\lfloor n \sqrt{x_b} \rfloor}{s} \cdot \binom{\lfloor n x_r / 2 \sqrt{x_b} \rfloor}{s}}{\binom{n}{2s}} = \lim_{n \to \infty} \frac{n^s x_b^{s/2} n^s x_r^s x_b^{-s/2}}{n^{2s} 2^s} \cdot \binom{2s}{s} =  \paren{\frac{x_r}{2}}^s \binom{2s}{s}. \]
    Consequently, $I_3 (F, (x_b, x_g)) \ge \paren{\frac{x_r}{2}}^s \binom{2s}{s}$.
\end{proof}

\section*{Appendix}

\begin{proof}[Proof of Lemma \ref{lem:1-W}]
    Let $F \in \C$ be the function that we want to prove $\delta$-good for $1 - W$. Using the fact that $G : [0,1] \to \R$ given by $x \mapsto F(1-x)$ from \cite[Lemma 3.1(ii)]{RW} is $\delta$-good for $W$ we find that \[\gamma = t (|, 1-W) = 1 - t (|, W) \text{ and } \eta = 1 - \sqrt{1 - \gamma}\] satisfy \begin{align*}
        \int_0^1 & F (d_{1-W} (x)) = \int_0^1 G (d_W (x)) \\
        &< \max \brac{ \paren{1 - \sqrt{1 - \gamma}} G (0) + \sqrt{1 - \gamma} G \paren{\sqrt{1 - \gamma}},~ \sqrt{\gamma} G(1 - \sqrt{\gamma}) + (1 - \sqrt{\gamma}) G(1) } + \delta \\
        &= \max \{ (1 - \eta) F(\eta) + \eta F(1),~ (1 - \sqrt{\gamma}) F (0) + \sqrt{\gamma} F (\sqrt{\gamma}) \} + \delta
    \end{align*} as desired.
\end{proof}

\begin{proof}[Proof of Lemma \ref{lem:lambdaW}]
    Let $F \in \C$ be any function that we want to prove $\delta$-good for $[\lambda, W]$. By \cite[Lemma 3.1(iii)]{RW}, the function $H: [0,1] \to \R$ given by $H(x) = F ((1 - \lambda))$ for all $x \in [0,1]$ is in $\C$. Thus it is $\delta$-good for $W$, which tells us that \[ \int_0^1 H (d_W (x)) ~dx < \max \{ (1 - \sqrt{\gamma}) H(0) + \sqrt{\gamma} H (\sqrt{\gamma}), (1 - \eta) H (\eta) + \eta H(1) \} + \delta, \] where $\gamma = t (|, W)$ and $\eta = 1 - \sqrt{1 - \gamma}$. Since \begin{align*}
        \int_0^1 F (d_{[\lambda, W]} (x)) ~dx &= \lambda F(0) + \int_\lambda^1 F \paren{(1 - \lambda) d_W \paren{\frac{x - \lambda}{1 - \lambda}}} ~dx \\
        &= \lambda F(0) + (1 - \lambda) \int_0^1 H (d_W (x)) ~dx,
    \end{align*} it follows that either \begin{align*}
        \int_0^1 F (d_{[\lambda, W]} (x)) ~dx &< \lambda F(0) + (1 - \lambda) ((1 - \sqrt{\gamma}) F(0) + \sqrt{\gamma} F ((1 - \lambda) \sqrt{\gamma}) + \delta) \\
        &= \lambda F(0) + (1 - \lambda) (1 - \sqrt{\gamma}) F(0) +  (1 - \lambda) \sqrt{\gamma} F ((1 - \lambda) \sqrt{\gamma}) + (1 - \lambda) \delta \\
        &\le \lambda F(0) + (1 - \lambda) (1 - \sqrt{\gamma}) F(0) +  (1 - \lambda) \sqrt{\gamma} F ((1 - \lambda) \sqrt{\gamma}) + \delta,
    \end{align*} or \begin{align*}
        \int_0^1 F (d_{[\lambda, W]} (x)) ~dx &< \lambda F(0) + (1 - \lambda) ((1 - \eta) F ((1 - \lambda) \eta) + \eta F (1 - \lambda) + \delta) \\
        &= \lambda F(0) + (1 - \lambda) (1 - \eta) F ((1 - \lambda) \eta) + (1 - \lambda) \eta F (1 - \lambda) + (1 - \lambda) \delta \\
        &\le \lambda F(0) + (1 - \lambda) (1 - \eta) F ((1 - \lambda) \eta) + (1 - \lambda) \eta F (1 - \lambda) + \delta.
    \end{align*} In the former case the right side simplifies to \[ \paren{1 - \sqrt{\gamma'}} F (0) + \sqrt{\gamma'} F \paren{\sqrt{\gamma'}} + \delta, \] where $\gamma' = (1 - \lambda)^2 \gamma = t (|, [\lambda, W])$, meaning that $F$ is, in particular, $\delta$-good for $[\lambda, W]$.

    So we may assume that the second alternative occurs. Setting $x = \lambda$, $y = (1 - \lambda) \eta$, and $z = (1 - \lambda) (1 - \eta)$ we thus get \[ \int_0^1 F (d_{[\lambda, W]} (x)) ~dx < x F(0) + z F(y) + y F(y+1) + \delta.\] Since $y^2 + 2yz = (1 - \lambda)^2 (2 \eta - \eta^2) = (1 - \lambda)^2 \gamma = \gamma'$, it follows in view of \cite[Lemma 3.3]{RW} that \[ \int_0^1 F (d_{[\lambda, W]} (x)) ~dx < \max \brac{\paren{1 - \sqrt{\gamma'}} F (0) + \sqrt{\gamma'} F \paren{\sqrt{\gamma'}},~ (1 - \eta') F(\eta') + \eta' F(1)} + \delta, \] where $\eta' = 1 - \sqrt{1 - \gamma'}$. This tells us that $F$ is indeed $\delta$-good for $[\lambda, W]$.
\end{proof}

\end{document}